\documentclass[12pt,leqno,draft]{amsart}
\usepackage[utf8]{inputenc}
\usepackage{amsmath, amssymb, amsthm, amsfonts}
\usepackage{enumitem}    
\usepackage[pdftex]{graphicx}

\theoremstyle{plain}
\newtheorem{thm}{Theorem}[section] 
 
\newtheorem{prop}[thm]{Proposition} 
\newtheorem*{cor}{Corollary} 

\theoremstyle{definition} 
\newtheorem{defn}{Definition}[section] 
\newtheorem*{exmp}{Example}
\newtheorem*{remark}{Remark}

\DeclareMathOperator{\Var}{Var}
\DeclareMathOperator{\Cov}{Cov}
\def\R{\mathbb{R}}

\def\N{\mathbb{N}}

\usepackage[authoryear]{natbib}
\begin{document}
\title[U-Statistic Processes for LRD Data]{Two-Sample U-Statistic Processes
for Long-Range Dependent Data}
\thanks{Research supported by Collaborative Research Project SFB 823 {\em Statistical Modeling of Nonlinear Dynamic Processes} of the German Research Foundation DFG}
\author[H. Dehling]{Herold Dehling}
\author[A. Rooch]{Aeneas Rooch}
\author[M.Wendler]{Martin Wendler}
\today
\address{Fakult\"at f\"ur Mathematik, Ruhr-Universit\"at Bochum, 44780 Bochum, Germany}
\email{herold.dehling@rub.de}
\email{aeneas.rooch@rub.de}
\email{martin.wendler@rub.de}


\begin{abstract}
Motivated by some common change-point tests, we investigate the asymptotic distribution of the U-statistic process 
\[
  U_n(t)=\sum_{i=1}^{[nt]} \sum_{j=[nt]+1}^n h(X_i,X_j), \; 0\leq t\leq 1,
\]
when the underlying data are long-range dependent. We present two approaches, one based on an expansion of the kernel $h(x,y)$ into Hermite polynomials, the other based on an empirical process representation of the U-statistic. Together, the two approaches cover a wide range of kernels, including all kernels commonly used in applications.
\end{abstract}
\maketitle

\section{Introduction}

In this paper, we will investigate the asymptotic distribution of the two-sample U-statistic process, defined
as
\begin{equation}
 U_n(t)=\sum_{i=1}^{[nt]} \sum_{j=[nt]+1}^n h(X_i,X_j), \; 0\leq t\leq 1,
\end{equation}
when the underlying data $(X_i)_{i\geq 1}$ are long-range dependent. We will present two general approaches to the analysis of the two-sample U-statistic process, one based on the Hermite expansion of 
the kernel $h(x,y)$, and the other on an empirical process representation. Combined, the two approaches cover a wide range of kernels, including all kernels commonly applied in change-point tests.

Two-sample U-statistic processes find applications in the detection of change-points in  a time series
$(X_i)_{i\geq 1}$. Many common test statistics  for changes in the location can be expressed in the form $\max_{1\leq k\leq n-1} \sum_{i=1}^k \sum_{j=k+1}^n h(X_i,X_j)$. This holds, for example, 
for the CUSUM test and the Wilcoxon change-point test, where $h(x,y)=y-x$ and $h(x,y)=1_{\{x\leq y\}}$,
respectively. The asymptotic distribution of such test statistics can be obtained if one knows the limit distribution of the corresponding two-sample U-statistic process.

The asymptotic distribution of the two-sample U-statistic process has been obtained earlier by Cs\"org\H{o} and Horv\'ath (1988), for i.i.d.\ data, and by Dehling, Fried, Garcia and Wendler (2013) for  short-range dependent data.  In both cases, $n^{-3/2}(U_n(t))_{0\leq t\leq 1}$ converges in distribution, on the space $D[0,1]$, towards a Gaussian process. In the case of long-range dependent data, the two-sample U-statistic process has been studied only for two specific examples, namely for $h(x,y)=y-x$ and 
$h(x,y)=1_{\{x\leq y\}}$; see \citet{horvath_kokoszka_1997} and Dehling, Rooch and Taqqu (2013a), respectively.

In this work, we consider a stationary Gaussian processes $(\xi_i)_{i\geq 1}$ with mean zero, variance 1 and auto-covariance function
\begin{align}
 \label{EQ:LRDcovariances}
 \gamma(k)=\Cov[\xi_1, \xi_{i+k}] = L(k) k^{-D},
\end{align}
with $0<D<1$ and a slowly varying function $L$. Note that $\gamma_k$ obeys a power law, while short-range dependent processes typically possess an auto-correlation function that decays exponentially fast. In long-range dependent time series, i.e.\ time series with such slowly decaying auto-correlations,  even observations in the distant past affect present behaviour; this is why long-range dependence is also called long memory.

Long-range dependence/long memory is an important issue in many fields: it can be detected in hydrology \citep{mandelbrot_wallis_1968} and in climate science \citep{caballero_jewson_brix_2002}. Moreover, it is an omnipresent property of data traffic both in local area networks and in wide area networks, and it can be explained by renewal processes that exhibit heavy-tailed interarrival distributions \citep{levy_taqqu_2000} or by the superposition of many highly variable sources. For a survey see \citet{willinger_etal_1996} and \citet{cappe_etal_2002}. In economics and finance, it is controversially discussed whether there is long-range dependence in economic data \citep{lo_1991}; volatilities can be long-range dependent processes \citep{breidt_etal_1998}, and sometimes there is evidence of long-range dependence in some stock market prices, while sometimes there is none \citep{willinger_taqqu_etal_1999}. \citet{baillie_1996} provides a survey of the major econometric work on long-range dependence, fractional integration and their application in economics. A short overview about probabilistic foundations and statistical models for long-range dependent data including extensive references is given by \citet{beran_2010}.

\section{Definitions and Main Results}
\subsection{Two-sample $U$-statistics processes}
Given two samples $X_1,\ldots,X_m$ and $Y_1,\ldots,Y_n$, and a kernel $h(x,y)$, the (non-normalized) two-sample U-statistic is defined as
\[
 U_{m,n}=\sum_{i=1}^m \sum_{j=1}^n h(X_i,Y_j).
\]
A well-known example is the Wilcoxon-Mann-Whitney test statistic for a difference in location between the two samples, which is obtained by taking $h(x,y)=1_{\{x\leq y \}}$.
Another example is the two-sample Gauss test statistic, which corresponds to $h(x,y)=y-x$.

In the present paper, we start from  a single sample $\xi_1,\ldots,\xi_n$, which is split into two pieces of consecutive observations $\xi_1,\ldots,\xi_{[n\lambda]}$, and $\xi_{[n\lambda]+1},\ldots,\xi_n$, for any $\lambda \in [0,1]$. 

\begin{defn}
Let $(\xi_i)_{i\geq 1}$ be a stochastic process, and let $h:\R^2\rightarrow \R$ be a measurable kernel. We define the two-sample U-statistic process
\begin{equation}
\label{EQ:2sample_Ustat}
U_{n}(\lambda)   =   \sum_{i = 1}^{[\lambda n]}\sum_{j = [\lambda n]+1}^{n} h(\xi_i, \xi_j),\; 0\leq \lambda\leq 1.
\end{equation}
We will view $(U_n(\lambda))_{0\leq \lambda\leq 1}$ as an element of the function space $D[0,1]$.
\end{defn}

In what follows, we shall first derive the asymptotic distribution of this process, in the case when $(\xi_i)_{i\geq 1}$ is a stationary Gaussian process with mean zero, variance 1, and autocovariance function (\ref{EQ:LRDcovariances}). The general case of a Gaussian subordinated process, i.e. $X_i=G(\xi_i)$, follows by considering the transformed kernel $h(G(x),G(y))$. 

For our results, we will usually require the kernels to satisfy moment conditions with respect to the standard normal distribution $\mathcal{N}
=\frac{1}{2\pi} e^{-x^2/2}e^{-y^2/2} dx dy$. We denote by $L^p(\R^2, \mathcal{N})$ the space of all $p$-integrable kernels. Observe that we may assume  without loss of generality that $h$ is centered, i.e. that $E[h(\xi, \eta)]=0$ for two i.i.d.\ Gaussian random variables $\xi, \eta \sim \mathcal{N}(0,1)$ -- otherwise just substract the mean.

One-sample U-statistics of LRD data have been studied by Dehling and Taqqu (1989, 1991) and by \citet{levyleduc_etal_2011}. The two-sample U-statistic process has only been studied for some specific kernels. For the kernel $h(x,y)=y-x$, the asymptotic distribution can be obtained directly from the functional non-central limit theorem of Dobrushin and Major (1979) and Taqqu (1979); see Dehling, Rooch and Taqqu (2013a) for details. For the kernel $h(x,y)=1_{\{x\leq y \}}$, Dehling, Rooch  and Taqqu (2013a) obtained the asymptotic distribution of the two-sample U-statistic process using the empirical process non-CLT of Dehling and Taqqu (1989).  As an application, they derived the asymptotic distribution of the Wilcoxon change-point test statistic for processes with LRD noise.

In this paper, we will derive the limit distribution of $(U_n(\lambda))_{0\leq \lambda \leq 1}$
for a broad class of kernels $h:\R^2\to\R$. We shall use two different approaches. The first approach uses a bivariate Hermite expansion of the kernel $h(x,y)$, while the second approach is based on an empirical process representation of the two-sample U-statistic process. In both cases, different technical assumptions regarding the kernel $h(x,y)$ have to be made. Together, the two approaches cover most examples that are relevant in statistics.

\subsection{A direct approach via the Hermite expansion}
\label{SEC:Direct_approach}
The motivation for the direct approach via the Hermite expansion arises from the study of partial sums of Gaussian subordinated processes. For any integer $k\geq 0$, we introduce the $k$-th order Hermite polynomial
\begin{equation}
  H_k(x)=(-1)^k e^{x^2/2} \frac{d^k}{dx^k} e^{-x^2/2}.
\label{eq:H-poly}
\end{equation}
It is well-known that the Hermite polynomials are orthogonal in the space $L^2(\R,N)$, and that the normalized Hermite polynomials $H_k(x)/\sqrt{k!}$ form an orthonormal basis for $L^2(\R,N)$. Thus, any $L^2$-function $h(x)$ can be expanded into a Hermite series
\begin{equation}
  h(x)=\sum_{k=0}^\infty \frac{a_k}{k!}H_k(x),
\label{eq:H-exp}
\end{equation}
where the coefficients are given by $a_k=E(h(\xi)H_k(\xi))$. Convergence in \eqref{eq:H-exp} is in the $L^2$-sense.

We define the Hermite rank $m$ of $h$ as the index of the lowest order non-vanishing term in the Hermite expansion, i.e.
\begin{equation}
m=\min\{k\geq 0: a_k\neq 0  \}.
\label{eq:H-rank}
\end{equation}
In this way, we may rewrite the Hermite expansion of $h$ as 
\begin{equation}
  h(x)=\sum_{k=m}^\infty \frac{a_k}{k!}H_k(x).
\label{eq:H-exp1}
\end{equation}
Thus, we obtain the following series expansion for the partial sum of the $h(\xi_i)$,
\[
 \sum_{i=1}^n h(\xi_i) =\sum_{k=m}^\infty \frac{a_k}{k!} \sum_{i=1}^n H_k(\xi_i).
\]
The variance of the different terms in this expansion depends crucially on the index $k$. In fact, one obtains
\begin{equation}
\label{EQ:Var(sum_Hl(Xj))}
 \Var\left[ \sum_{i=1}^{n} H_{k}(\xi_{i})\right]  \sim   
\begin{cases}
  n k! C,  & \text{if }D k > 1,\\
  n^{2-D k} L(n)^{k} \frac{2k!}{(1-D k)(2-D k)}, & \text{if }D k < 1.
\end{cases} 
 \end{equation}
Moreover, for different indices $k,l$, the random variables $H_k(\xi_i)$ and $H_l(\xi_j)$ are uncorrelated, and thus 
\[
  \Var\left(\sum_{k=m}^\infty \frac{a_k}{k!} \sum_{i=1}^n H_k(\xi_i)\right) =
  \sum_{k=m}^\infty \frac{a_k^2}{(k!)^2} \Var\left(\sum_{i=1}^n H_k(\xi_i)\right).
\] 
When $mD<1$, this expansion is dominated by the lowest order term. In fact, defining
\begin{equation}
d_n^2=d_n(m)^2=\Var\left( \sum_{i=1}^n H_m(\xi_i)\right) \sim c_mn^{2-Dm} L^m(n),
\label{eq:d_n}
\end{equation}
where
\begin{equation}
 c_m=\frac{2m!}{(1-Dm)!(2-Dm)! },
\label{eq:c_m}
\end{equation}

Taqqu (1975) proved the reduction principle, stating that 
\[
  \frac{1}{d_n} \left|\sum_{i=1}^n h(\xi_i) - \frac{a_m}{m!} \sum_{i=1}^m H_m(\xi_i)\right| 
 \rightarrow 0,
\]
in probability. Thus the study
 of the partial sums of arbitrary functionals of an LRD Gaussian process  can be reduced to the study of partial sums of Hermite polynomials $\sum_{i=1}^n H_m(\xi_i)$. Weak convergence of the latter sums has been studied by Taqqu (1979), and independently by Dobrushin and Major (1979). These authors proved 
\begin{equation}
 \frac{1}{d_n} \sum_{i=1}^{[nt]} H_m(\xi_i) \rightarrow Z_m(t),
\label{eq:non-CLT}
\end{equation}
where $Z_m(t)$ denotes the $m$-th order Hermite process. For details, see e.g. the forthcoming monograph by Pipiras and Taqqu (2014).

Motivated by the results of Taqqu (1977, 1979) for partial sums of Gaussian subordinated processes, we will now study the Hermite expansion of functions $h\in L^2(\R^2,\mathcal{N})$. We define the two-dimensional Hermite polynomials $H_{kl}(x,y) = H_{k}(x) H_{l}(y)$,
where $H_k(x)$  is the one-dimensional Hermite polynomial, as defined in \eqref{eq:H-poly}. Then,
\[
  \frac{H_{k}(x) H_l(y)}{\sqrt{k!l!}}, \, k\geq 0, l\geq 0,
\]
is an orthonormal basis for the Hilbert space $L^2(\R^2,\mathcal{N})$. Thus, we obtain for any $h\in L^2(\R^2,\mathcal{N})$ the Hermite expansion
\begin{equation}
\label{EQ:HermiteExpansion1} 
h(x,y) = \sum_{k,l=0}^\infty \frac{a_{kl}}{k!\, l!} H_{k}(x) H_{l}(y),
\end{equation}
where the coefficients are given by
\begin{equation}
a_{kl}   =   E\left[ h(\xi, \eta) H_k(\xi) H_l(\eta) \right]   =   \int_{\R^2} h(x, y) H_k(x) H_l(y) \, \varphi(x) \varphi(y) \;dx \, dy.
\label{eq:H-coeff}
\end{equation}
Here, $\varphi(x)$ denotes the one-dimensional standard normal probability density function.
Note that \eqref{EQ:HermiteExpansion1} is an expansion in the Hilbert space $L^2(\R^2,\mathcal{N})$, i.e. the series 
\eqref{EQ:HermiteExpansion1} converges in $L^2(\R^2,\mathcal{N})$ towards the function $h(x,y)$.

We now order the terms in the expansion \eqref{EQ:HermiteExpansion1} according to the size of $k+l$:
\begin{equation}
\label{EQ:HermiteExpansion2} 
h(x,y) = \sum_{q=m}^{\infty} \sum_{\substack{k,l: \\ k+l = q}} \frac{a_{kl}}{k!\, l!} H_{k}(x) H_{l}(y),
\end{equation}
where $m$ is the smallest integer for which there exists a non-zero Hermite coefficient $a_{kl}$ with $k+l=m$.
\begin{defn}
The \textit{Hermite rank} of a function $h\in L^2(\R^2,\mathcal{N})$ is defined as
\[
m = \inf\{ k+l \, | \, k,l \geq 0, a_{kl} \neq 0 \},
\]
where $a_{kl}$ is the coefficient in the Hermite expansion \eqref{EQ:HermiteExpansion1}.
\end{defn}

For an alternative approach and different notions of a two-dimensional Hermite rank, see the work by \citet{levyleduc_taqqu_2013}.

\begin{defn}

Let $\xi, \eta \sim\mathcal{N}(0,1)$ be two independent standard normal random variables. We define
\[
\mathcal{G}^1(\R^2, \mathcal{N}) = \{G:\R^2\rightarrow \R \text{ integrable} \; | \; E[G(\xi, \eta)]=0\} \subset L^1(\R^2, \mathcal{N}),
\] 
the class of (with respect to the standard normal measure) centered and integrable functions on $\R^2$, and
\[
\mathcal{G}^2(\R^2, \mathcal{N}) = \{G:\R^2\rightarrow \R \text{ measurable} \; | \; E[G(\xi, \eta)]=0, E[G^2(\xi, \eta)]=1\} \subset L^2(\R^2, \mathcal{N}),
\] 
the class of (with respect to the standard normal measure) centered, normalized and square-integrable functions on $\R^2$. Analogously, we define the class $\mathcal{G}^2(\R,\mathcal{N})$.
\end{defn}
Any function $G:\R^2\rightarrow \R$ which is measurable with mean zero and finite variance under standard normal measure can be normalized by dividing the standard deviation, so it can be considered as a function in $\mathcal{G}^2 = \mathcal{G}^2(\R^2, \mathcal{N})$.

\begin{thm}
\label{THM:2sampleUstatistics_dep} 
Let $(\xi_i)_{i\geq 1}$ be a stationary Gaussian process with mean 0, variance 1 and covariances \eqref{EQ:LRDcovariances}. Let $D m < 1$ and let $h \in \mathcal{G}^2(\R^2, \mathcal{N})$ be a function with Hermite rank $m$ whose Hermite coefficients satisfy
\begin{equation}
\label{eq:akl_sum}  
\sum_{k, l} \frac{\vert a_{kl} \vert}{\sqrt{k!\, l!}} < \infty.
\end{equation}
Then as $n\to\infty$
\[
 \frac{1}{d'_n \, n} \left\vert \sum_{i = 1}^{[\lambda n]}\sum_{j = [\lambda n]+1}^{n} h(\xi_i, \xi_j) - \sum_{\substack{k,l: \\ k+l = m}} \frac{a_{kl}}{k!\, l!} \sum_{i = 1}^{[\lambda n]}\sum_{j = [\lambda n]+1}^{n} H_{k}(\xi_{i}) H_{l}(\xi_{j})  \right\vert \stackrel{L^{1}}{\longrightarrow} 0
\]
uniformly in $\lambda \in [0,1]$ and
\begin{equation}
\label{EQ:Limit_in_Thm_UStatistics_dep} 
 \left(\frac{1}{d'_n \, n} \sum_{i = 1}^{[\lambda n]}\sum_{j = [\lambda n]+1}^{n} h(\xi_i, \xi_j)\right)_{0\leq \lambda \leq 1} \stackrel{\mathcal{D}}{\longrightarrow} \left(\sum_{\substack{k,l: \\ k+l = m}} 
 \frac{a_{kl}}{k!\, l!} (c_k c_l)^{1/2} Z_k(\lambda)(Z_l(1)-Z_l(\lambda))\right)_{0\leq \lambda \leq 1}
\end{equation}
in $D[0,1]$, where 
\begin{equation}
\label{EQ:dn'^2}
d'^2_n = n^{2-mD} L(n)^m
\end{equation}
and the $(Z_k(\lambda))_{\lambda \geq 0}$, $k=0, \ldots, m$, are dependent processes which can be expressed as $k$-fold Wiener-It\={o}-Integrals,
\begin{equation}
 Z_k (\lambda)   =   K^{-k/2} c_k^{-1/2} \int'_{\R^k} \frac{e^{i\lambda \sum_{j=1}^{k} x_j}-1}{i \sum_{j=1}^{k} x_j} \left( \prod _{j=1}^{k} |x_j|^{(D-1)/2}\right) \; dW(x_1) \ldots dW(x_k),  \label{eq:Z-Rep}  
\end{equation}
where $c_k$ is defined in \eqref{eq:c_m} and where
$ K = \int_{\R} e^{ix} |x|^{D -1}\; dx = 2 \Gamma(D) \cos(D \pi / 2)$.
\end{thm}
\begin{remark}
(i) Formula \eqref{eq:Z-Rep} denotes the multiple Wiener-It\={o} integral with respect to the random spectral measure $W$ of the white-noise process, where ${\int}'$ means that the domain of integration excludes the hyperdiagonals $\{x_i = \pm x_j,\, i\neq j\}$, see \citet{major_1981b} or also \citet[p.\ 1769]{dehling_taqqu_1989}. The constant of proportionality $c_m$ ensures that $E[Z_m(1)]^2=1$. \citet{taqqu_1979} and \citet[Chap.\ 3.2]{pipiras_taqqu_2011} give another representation. 

(ii) The process $(Z_m(\lambda))_{\lambda \geq 0}$ in \eqref{eq:Z-Rep} is called Hermite process of order $m$. For $m=1$, this is a standard fractional Brownian motion.  When $m\geq 2$, the process 
$(Z_m(\lambda))_{\lambda \geq 0}$ is non-Gaussian.

(iii) The process $(Z_m(\lambda))_{\lambda \geq 0}$ is self-similar with parameter
\[
 H   =   1-\frac{Dm}{2} \in \left(\frac{1}{2}, 1\right),
\]
i.e.\ $(Z_m(c \lambda))_{\lambda \geq 0}$ and $c^H (Z_m(\lambda))_{\lambda \geq 0}$ have the same finite-dimensional distributions for all constants $c>0$.

(iv) Conditions of the type \eqref{eq:akl_sum} are not uncommon in the study of U-statistic of dependent data. E.g., one finds such conditions in the recent papers by Neumann and Leucht (2013) and Denker and Gordin (2013). 

(v) The scaling factor \eqref{EQ:dn'^2} differs slightly from the usual scaling \eqref{eq:d_n} in that it does not include the normalizing constant $c_m$. This is caused by the fact that the limit now is a linear combination of two possibly different Hermite processes $Z_k$, $Z_l$ and thus the associated factors $c_k$, $c_l$ cannot be divided out and must remain inside the sum of the right-hand side of \eqref{EQ:Limit_in_Thm_UStatistics_dep}.
\end{remark}

For the most interesting and simple case $m=1$, we can give a handy explicit representation of the limit \eqref{EQ:Limit_in_Thm_UStatistics_dep}, because then $Z_1(\lambda)$ is standard fractional Brownian motion $B_H(\lambda)$ with $H=1-D/2$. 

\begin{cor}
If the Hermite rank of $h(x,y)$ is $m=1$, the statement of Theorem~\ref{THM:2sampleUstatistics_dep} simplifies to 
\begin{equation}
\label{EQ:Limit_in_Thm_UStatistics_dep_m=1} 
 \frac{1}{d'_n \, n} \sum_{i = 1}^{[\lambda n]}\sum_{j = [\lambda n]+1}^{n} h(\xi_i, \xi_j)   \stackrel{\mathcal{D}}{\longrightarrow}   \sqrt{c_1} \left(a_{1,0} (1-\lambda) B_H(\lambda) + a_{0,1} \lambda (B_H(1)-B_H(\lambda))\right),
 \end{equation}
where $B_H(\lambda)$ is fractional Brownian motion with parameter $H=1-D/2$. 
\end{cor}

\begin{proof}
By Theorem~\ref{THM:2sampleUstatistics_dep}, the limit is 
\begin{align*}
&\hspace{-0.8cm}\sum_{k+l = 1} \frac{a_{kl}}{k!\, l!} \sqrt{c_k c_l} Z_k(\lambda)(Z_l(1)-Z_l(\lambda))   \\
   &=   a_{1,0} \sqrt{c_1} Z_1(\lambda) \left(Z_0(1) - Z_0(\lambda)\right ) + a_{0,1} \sqrt{c_1} Z_0(\lambda) \left(Z_1(1) - Z_1(\lambda)\right ) \\
   &=   a_{1,0} \sqrt{c_1} (1-\lambda) B_H(\lambda) + a_{0,1} \sqrt{c_1} \lambda (B_H(1) - B_H(\lambda))
\end{align*}
with $c_1 = 2/((1-D)(2-D))$, because $Z_0(t)=t$ and $Z_1(t) = B_{H}(t)$.
\end{proof}

In Section~\ref{SEC:Examples} we illustrate this by some examples. Since the approach is subject to technical restrictions which are sometimes difficult to check, we will develop some easily verifiable criteria for it in Section~\ref{SEC:Handy_Criteria}. Unfortunately, the technical restrictions are not satisfied by some special kernels like the Wilcoxon kernel $h(x,y) = I_{\{x \leq y\}}$. Thus, in Section~\ref{SEC:Approach_via_empirical_process}, we present an approach using an empirical process representation of the two-sample U-statistic process. 

\subsection{An approach via empirical processes}
Our second approach to the study of the asymptotic distribution of two-sample U-statistic processes uses a representation of $U_n(\lambda)$ as a functional of the empirical distribution function. This approach has been used earlier for one-sample U-statistics by Dehling and Taqqu (1991); see also  \cite{beutner_zaehle_2013} for some recent extensions. 
\label{SEC:Approach_via_empirical_process}
\citet{dehling_taqqu_1989} have proved a limit theorem for the two-parameter empirical process $(F_{[\lambda n]}(x)-F(x))_{x\in [-\infty, \infty], \lambda \in [0,1]}$, where
\[
 F_k (x)  =  \frac{1}{k} \sum_{i=1}^k I_{\{G(\xi_i) \leq x\}}
\]
denotes the empirical distribution function (e.d.f.) of the first $k$ observations $G(\xi_1), \ldots, G(\xi_k)$ and $F$ denotes the cumulative distribution function (c.d.f.) of the $G(\xi_i)$. They consider the Hermite expansion
\[
 I_{\{G(\xi_i) \leq x\}} - F(x) = \sum_{k=1}^\infty \frac{J_k(x)}{k!} H_k(\xi_i),
\]
where $H_k$ again denotes the $k$-th Hermite polynomial and $J_k(x)$ is the $k$-th Hermite coefficient in this expansion,
\begin{equation}
\label{eq:J_k} 
 J_k(x) = E\left[ I_{\{G(\xi) \leq x\}} H_k(\xi)\right] = (2\pi)^{-1/2} \int_{-\infty}^{\infty} I_{\{G(s) \leq x\}} H_k(s) e^{-s^2/2} \; ds
\end{equation}
with $\xi \sim \mathcal{N}(0,1)$. 

\begin{defn}[Hermite rank of class of functions]
We define the \textit{Hermite rank of the class of functions}
$\{I_{\{G(\xi_i)\leq x\}}-F(x)$, $x\in \R\}$ by
\begin{equation}
\label{EQ:Hermite_rank_functionclass}
 m=\min\{k\geq 1: J_k(x)\neq 0 \mbox{ for some } x\in \R  \}.
\end{equation}
\end{defn}

Now, we can state the second main result of the present paper.
\begin{thm}
\label{THM:Thm_general_kernels}
Let $(\xi_i)_{i\geq 1}$ be a stationary Gaussian process with 
mean zero, variance $1$ and auto-covariance function as in \eqref{EQ:LRDcovariances}.
Moreover, let $G:\R\rightarrow \R$ be a measurable function with $E[G(\xi_i)]=0$, and define
\[
X_k=G(\xi_k).
\]
Assume that $X_k$ has a continuous distribution function $F$. Let $m$ denote the Hermite rank of the class of functions 
$I_{\{G(\xi_i)\leq x\}}-F(x)$, $x\in \R$, and assume that $mD<1$, where $D$ is the exponent in 
\eqref{EQ:LRDcovariances}.
Moreover, define
\begin{equation}
\label{EQ:h_tilde}
\tilde{h}(x) := \int h(x, y)\; dF(y),
\end{equation}
and assume that for some constant $c\in (0,\infty)$,
\begin{eqnarray}
\| h(\cdot,y)\|_{TV} &\leq& c 
\label{eq:h1-tv}\\
\| h(x,\cdot)\|_{TV} &\leq& c.
\label{eq:h2-tv}
\end{eqnarray}
Then
\[
\left( \frac{1}{n \, d_n} \sum_{i=1}^{[\lambda n]} \sum_{j=[\lambda n]+1}^{n} \left( h(X_i, X_j) - \iint h(x, y)\; dF(x)dF(y)\right) \right)_{0 \leq \lambda \leq 1}
\]
converges in distribution towards the process
\begin{equation}
\label{EQ:Thm_general_kernels_limit}  
\Big( -(1-\lambda)Z(\lambda) \int J(x) \; d\tilde{h}(x)  
  - \lambda(Z(1)-Z(\lambda)) \int\! \!\left( \int J(y)\; dh(x, y)(y) \right)  dF(x)\Big)_{0 \leq \lambda \leq 1}.
\end{equation}
Here $Z(\lambda) = Z_m(\lambda)/m!$, where $(Z_m(\lambda))_{\lambda \geq 0}$ denotes the $m$-th order Hermite process defined in \eqref{eq:Z-Rep}, and $J(x)=J_m(x)$.
\end{thm}

\section{Summability criteria for Hermite coefficients}
\label{SEC:Handy_Criteria}

For most kernels, the summability condition \eqref{eq:akl_sum} is not easily verified, as the Hermite coefficients cannot be explicitly calculated. In this section, we will derive alternative criteria that imply  \eqref{eq:akl_sum}. Recall that the Fourier transform of a function  $f\in L^2(\R^d)$ is defined by
  \begin{align}
\label{EQ:Def_FourierTransform}
\mathcal{F}(f)(\xi)=\hat{f}(\xi) = \frac{1}{(2\pi)^{d/2}} \int_{\R^d} f(x) e^{- i x\cdot \xi} \; dx.
\end{align}

\begin{prop}
\label{LEM:FourierTransform_SumFinite} Let $h\in L^2(\R^2)$. Then condition \eqref{eq:akl_sum} is fulfilled
if
\begin{equation}
\int |\mathcal{F}(h)(s,t)| (1+s^2)(1+t^2) ds dt <\infty.
\label{eq:FT_Int}
\end{equation}
A sufficient condition for \eqref{eq:FT_Int} to hold is that 
$\frac{\partial^8}{\partial x^4\partial y^4} h(x,y)$ exists and is in $L^1(\R^2)$.
\end{prop}

\begin{proof}
By the Plancherel theorem we can write the Hermite coefficients in the following way:
\begin{align*}
a_{kl}   &=   \frac{1}{2\pi} \iint_{\R^2} h(x,y) H_k(x)H_l(y) e^{-(x^2+y^2)/2} \; dx \,dy \\
   &=   \frac{1}{2\pi} \iint_{\R^2} \hat{h}(s,t) g(s,t) \; ds \,dt,
\end{align*}
where $g(s,t) := \mathcal{F}(H_k(x)H_l(y) e^{-(x^2+y^2)/2})$.  In order to give an explicit representation of $g(s,t)$, we will apply some well-known properties of Fourier transforms and Hermite polynomials, namely
\begin{eqnarray*}
    \mathcal{F} \left( e^{-(x^2+y^2)/2} \right) &=& e^{-(s^2+t^2)/2} \\
 \mathcal{F}\left(\frac{\partial^{k+l}}{\partial x^{k} \partial y^{l}} f(x,y) \right) &=&  i^{k+l} s^{k} t^{l} \hat{f}(s,t)\\
  H_k(x)H_l(y) e^{-(x^2+y^2)/2} &=&(-1)^{k+l} \frac{\partial^{k+l}}{\partial x^{k} \partial y^{l}} e^{-(x^2+y^2)/2}. 
\end{eqnarray*}
With these formulae we can write 
\[
a_{kl}   =   \frac{1}{2\pi} \iint_{\R^2} \hat{h}(s,t) (-i)^{k+l} s^{k} t^{l} e^{-(s^2+t^2)/2} \; ds \,dt,
\]
and hence we obtain
\begin{eqnarray}
\sum_{k,l} \frac{|a_{kl}|}{\sqrt{k!}\sqrt{l!}} 
&=& \frac{1}{2\pi} \sum_{k,l} \left|  \iint_{\R^2} \hat{h}(s,t) (-i)^{k+l} \frac{s^{k} t^{l}}{\sqrt{k!}\sqrt{l!}} e^{-(s^2+t^2)/2} \; ds \,dt \right| \nonumber\\
&\leq &  \frac{1}{2\pi} \sum_{k, l} \iint_{\R^2} |\hat{h}(s,t)| \frac{|s|^{k} |t|^{l}}{\sqrt{k!} \sqrt{l!}} e^{-(s^2+t^2)/2} \; ds \,dt \nonumber\\
&=& 
\frac{1}{2\pi} \iint_{\R^2} |\hat{h}(s,t)|\left( \sum_{k=1}^\infty \frac{|s|^{k}}{\sqrt{k!}} \right)
 \left( \sum_{l=1}^\infty \frac{|t|^{l}}{\sqrt{l!}} \right)  e^{-(s^2+t^2)/2} \; ds \,dt.
\label{eq:akl-sum2}
\end{eqnarray}
Now, we will bound the series $\sum_{k=1}^\infty \frac{s^k}{\sqrt{k!}}$, for $s\geq 0$. For notational 
convenience, we write $(k/2)!:=\Gamma(k/2+1)$. Using Stirling's formula for the Gamma function, we then obtain
$(k/2)!/\sqrt{k!}\sim C 2^{-k/2} k^{1/4} \leq C k\, 2^{-k/2}  $, and thus
\[
 \sum_{k=1}^\infty \frac{s^k}{\sqrt{k!}} 
  \leq  C\sum_{k=1}^\infty k \frac{s^k  2^{-k/2}}{(k/2)!} 
  =  C \sum_{k=1}^\infty k \frac{(s^2/2)^{k/2}}{(k/2)!}.  
\]
Now we use the series expansion
\begin{equation}
\sum_{k=1}^{\infty} \frac{x^{k/2} k/2}{(k/2)!}   =   \frac{\sqrt{x}}{\sqrt{\pi}} + xe^{x}(1+\operatorname{erf}(\sqrt{x})),
\label{eq:erf-series}
\end{equation}
where
$\operatorname{erf}(x)= 2/\sqrt{\pi} \int_0^x e^{-t^2}\, dt$ denotes the Gaussian error function. 
The identity \eqref{eq:erf-series} follows from the series expansion
$\operatorname{erf}(z)=\frac{2}{\sqrt{\pi}} e^{-z^2} \sum_{k=0}^\infty \frac{2^k}{1\cdot 3 \cdots (2k+1)}z^{2k+1}$; see Gradshteyn and Ryzhik (1980), p. 931. Using the fact that the Gaussian error function is bounded by $1$, we obtain
\[ 
 \sum_{k=1}^\infty \frac{s^k}{\sqrt{k!}} \leq C\left( s+s^2 e^{s^2/2} \right),
\] 
and hence
\begin{eqnarray*}
 \sum_{k, l} \frac{|a_{kl}|}{\sqrt{k! \, l!}}   
&\leq&   C \iint_{\R^2} |\hat{h}(s,t)| \left( s e^{-s^2/2} + s^2 \right) \left( t e^{-t^2/2} + t^2 \right) \; ds \,dt \\
&\leq & C \iint_{\R^2} |\hat{h}(s,t)| \left( 1+ s^2 \right) \left( 1 + t^2 \right) \; ds \,dt.
\end{eqnarray*}
Since $\mathcal{F}(\frac{\partial^8}{\partial^4 x \partial^4 y} h(x,y))=s^4t^4 \mathcal{F}(h)(s,t)$, and since 
$\frac{\partial^8}{\partial^4 x \partial^4 y} h(x,y) \in L^1$ by assumption, we obtain from the Riemann-Lebesgue lemma that $s^4t^4 \mathcal{F}(h)(s,t) \rightarrow 0$ as $s,t\rightarrow \infty$. Thus 
$\mathcal{F}(h)(s,t)=o(\frac{1}{s^4t^4})$, and hence the integral \eqref{eq:FT_Int} is finite.
\end{proof}

\begin{exmp} Any function $h \in L^1(\R^2, \lambda)$ whose derivatives of order 8 are integrable  satisfies the summability condition \eqref{eq:akl_sum}, for instance:
\renewcommand{\labelenumi}{(\roman{enumi})}
\begin{enumerate}
 \item a (normalized) Hermite function
\[
 \tilde{h}_{kl}(x,y)   =   \frac{1}{\sqrt{2^{k+l} k!\, l! \pi}} H_{kl}(x,y) e^{-(x^2+y^2)/2}
\]
\item a Gaussian function
\begin{align*}
g(x,y) = a \exp \left\lbrace  -b \cdot \left( \begin{pmatrix} x \\ y \end{pmatrix} -\begin{pmatrix} \mu_1 \\ \mu_2 \end{pmatrix}\right) ^{t} \Sigma^{-1} \left( \begin{pmatrix} x \\ y \end{pmatrix} -\begin{pmatrix} \mu_1 \\ \mu_2 \end{pmatrix}\right)\right\rbrace 
\end{align*}
with $a,b,\mu_1, \mu_2 \in \R$ and $\Sigma \in \R^{2\times 2}$ a symmetric positive-definite matrix
\item a smooth function with bounded support like the bump function
\[ 
f(x,y) = \begin{cases} 
e^{\frac{-1}{1-x^2}}e^{\frac{-1}{1-y^2}} & |x|,|y| < 1 \\ 
0   &   \text{else} 
\end{cases}
\]
\end{enumerate}
\end{exmp}

\section{Examples}
\label{SEC:Examples}
\subsection{Examples related to the Hermite expansion approach}

\subsubsection{CUSUM statistic}

The kernel $h(x,y)=x-y$ leads to the \textit{CUSUM statistic}
\[
 U_{C, n} (\lambda)   =   \sum_{i=1}^{[\lambda n]} \sum_{j=[\lambda n]+1}^{n} (\xi_i- \xi_j) =   [\lambda n](n-[\lambda n]) \left( \frac{1}{[\lambda n]}\sum_{i=1}^{[\lambda n]} X_i 
-\frac{1}{n-[\lambda n]}\sum_{i=[\lambda n]+1}^n X_i \right),
\]
a scaled difference of the means of the first and the second part of the sample. The kernel $h(x,y)=x-y$ is of course in $L^2(\R^2,\mathcal{N})$ and its Hermite expansion can be read off without calculating: 
\[
h(x,y) = x-y = \frac{a_{1,0}}{1! 0!} H_1(x) + \frac{a_{0,1}}{0! 1!} H_1(y),
\]
so its Hermite coefficients are
 \[ a_{kl}   = 
 \begin{cases}
  1   &   k=1, l=0 \\
 -1   &   k=0, l=1 \\ 
 0   &   \text{else}
 \end{cases},
 \]
and condition \eqref{eq:akl_sum} is trivially fulfilled. 
Thus, the Corollary to Theorem~\ref{THM:2sampleUstatistics_dep}  yields

\[
\frac{1}{d'_n \, n}  U_{C, n} (\lambda)   \stackrel{\mathcal{D}}{\longrightarrow}   \sqrt{c_1} \left( (1-\lambda) B_H(\lambda) - \lambda (B_H(1)-B_H(\lambda)) \right), 
\]
so we have reproduced the result of \citet{horvath_kokoszka_1997}.

So far, we have considered Gaussian observations $\xi_1, \ldots, \xi_n$. As mentioned in the introduction, all results can be extended to general data $G(\xi_1), \ldots, G(\xi_n)$, where $G \in \mathcal{G}^2(\R, \mathcal{N}) \subset L^2(\R, \mathcal{N})$ is a transformation like the quantile transformations in the 
papers of \citet{dehling_rooch_taqqu_2013a, dehling_rooch_taqqu_2013b}, by considering the kernel $h(G(x), G(y))$ instead of $h(x,y)$. This is what we will do now exemplarily. The Hermite coefficients of the function $h(G(x), G(y))$ are
\begin{align*}
a_{kl}   &=   \iint_{\R^2}(G(x)-G(y)) H_k(x) H_l(y) \; d\Phi(x) \, d\Phi(y)  \\
   &=   \int_\R G(x) H_k(x) \; d\Phi(x) \cdot \int_\R H_l(y)\; d\Phi(y)  -  \int_\R G(y) H_l(y) \; d\Phi(y) \cdot \int_\R H_k(x)\; d\Phi(x)   \\
   &=   \begin{cases}
      0	&	\text{if } k,l \neq 0\\
      -a_l	&	\text{if } k=0, l\neq 0\\
      a_k	&	\text{if } k\neq 0, l= 0
   \end{cases},
\end{align*}
where $a_p = E[G(\xi) H_p(\xi)]$ is the $p$-th Hermite coefficient of $G$. Thus for $G:\R\to\R$, $G \in \mathcal{G}^2(\R, \mathcal{N})$ and $h(x,y)=x-y$, the summability condition \eqref{eq:akl_sum} turns into a similar condition on the transformation $G$: 
\begin{align*}
\sum_{k,l} \frac{|a_{kl}|}{\sqrt{k! \, l!}}   &=   2 \sum_{k=1}^\infty \frac{|a_k|}{\sqrt{k!}}<\infty
\end{align*}

\subsubsection{The Wilcoxon statistic} 

The kernel $h(x,y)=I_{\lbrace x\leq y \rbrace}$ yields the well-known \textit{Mann-Whitney-Wilcoxon statistic}
\[ 
U_{W, n} (\lambda) = \sum_{i=1}^{[\lambda n]} \sum_{j=[\lambda n]+1}^{n} I_{\lbrace \xi_i\leq \xi_j \rbrace}.
\]
We will now show that this kernel does not fulfil the summability condition \eqref{eq:akl_sum}. Nevertheless, if we ignore this, the above theorems reproduce the results from \citet{dehling_rooch_taqqu_2013a} which suggests that condition \eqref{eq:akl_sum} may be too strong and Theorem~\ref{THM:2sampleUstatistics_dep} may hold under milder assumptions.

That $h(x,y)=I_{\lbrace x\leq y \rbrace}$ does not satisfy \eqref{eq:akl_sum} is neither intuitively visible nor enjoyable to verify. One can show, using integration by parts and identities for Hermite polynomials and the Gamma function, that the Hermite coefficients of $h(x,y)=I_{\lbrace x\leq y\rbrace}$ are
\begin{align}
\label{EQ:HermitCoeff_Ix<y_general}
 a_{k,l}   = 
 \begin{cases}
  \frac{(-1)^{\frac{l+3k-1}{2}}}{2\pi} \Gamma\left( \frac{l+k}{2}\right)   &   l+k \text{ odd and positive} \\ 
 0   &   l+k \text{ even and positive} \\ 
 \frac{1}{2}   &   l=k=0 
 \end{cases}.
\end{align}
In order to show that $\sum_{k,l=1}^{\infty} |a_{k,l}|/\sqrt{k! \, l!}$ diverges, it is enough to consider the first odd diagonal where $l=k+1$, because there we have already with Sterling's approximation
\[
\frac{|a_{k,l}|}{\sqrt{k! \, l!}}  \sim   \frac{(2k-1)^{k} e}{2^{k}(k+1)^{k/2+3/4}k^{k/2+1/4}}   =   \frac{\left (1-\frac{1}{2k}\right )^{k} \frac{1}{k} e}{\left (1+\frac{1}{k}\right )^{k/2} \left (1+\frac{1}{k}\right )^{3/4}}   \sim   \frac{1}{k}.
\]

Let us for a moment ignore that the Wilcoxon kernel does not fulfill the summability condition \eqref{eq:akl_sum}, which may be too rigorous anyway, and apply Theorem~\ref{THM:2sampleUstatistics_dep}. To this end, we use \eqref{EQ:HermitCoeff_Ix<y_general} or calculate the first Hermite coefficients manually:
 \begin{align*}
  a_{0,0}   &=   \iint_{\lbrace x\leq y\rbrace} H_{0}(x) H_{0}(y) \varphi(x) \varphi(y) \; dx\, dy    =   \iint_{\lbrace x\leq y\rbrace} \varphi(x) \varphi(y) \; dx\, dy   =   \frac{1}{2} \\
  a_{1,0}   &=   \iint_{\lbrace x\leq y\rbrace} x \varphi(x) \varphi(y) \; dx\, dy    =   -\frac{1}{2\sqrt{\pi}} \\ 
  a_{0,1}   &=   \iint_{\lbrace x\leq y\rbrace} y \varphi(x) \varphi(y) \; dx\, dy    =   \frac{1}{2\sqrt{\pi}}
\end{align*} 
Since we formulated the theorem for centered kernels, we consider $h(x,y) - E[h(\xi, \eta)] = I_{\{x\leq y\}}-1/2$, which has Hermite rank $m=1$. So the Corrolary to Theorem~\ref{THM:2sampleUstatistics_dep} states that 
\begin{align*}
\frac{1}{n\, d'_n} \sum_{i=1}^{[\lambda n]} \sum_{j=[\lambda n]+1}^{n} I_{\lbrace \xi_i\leq \xi_j \rbrace}   &\stackrel{\mathcal{D}}{\longrightarrow}   \sqrt{c_1} \left( a_{1,0}(1-\lambda) B_H(\lambda) + a_{0,1} \lambda (B_H(1)-B_H(\lambda)) \right) \\
   &=   \frac{\sqrt{c_1}}{2\sqrt{\pi}} \left( \lambda B_H(1) - B_H(\lambda) \right).
\end{align*}

Bearing in mind that $\int_\R J_1(x) \, d\Phi(x) = -(2\sqrt{\pi})^{-1}$, we have just reproduced the results of \citet{dehling_rooch_taqqu_2013a} for the Gaussian case.

\subsection{Examples related to the empirical process approach}

\subsubsection{The kernel $h(x,y)=I_{\lbrace x\leq y \rbrace}$}
\citet{dehling_rooch_taqqu_2013a} have investigated the asymptotic distribution of the \textit{Mann-Whitney-Wilcoxon statistic}
\[ 
W_{[\lambda n], n} = \frac{1}{n\, d_n} \sum_{i=1}^{[\lambda n]} \sum_{j=[\lambda n]+1}^{n} \left( I_{\lbrace X_i\leq X_j \rbrace} - \frac{1}{2} \right).
\]
We will now show that the conditions of Theorem~\ref{THM:Thm_general_kernels} are satisfied, and that we obtain the same result as \citet{dehling_rooch_taqqu_2013a}. For fixed $x$, the kernel 
$h(x,y)=1_{\{x\leq y \}}$ is of bounded variation as a function of $y$, and the same holds for fixed $y$. In fact \eqref{eq:h1-tv} and \eqref{eq:h2-tv} are satisfied with $c=1$. Moreover, we obtain
\begin{eqnarray*}
\tilde{h}(x)   &=&   \int h(x, y)\, dF(y)   =   \int_{x}^\infty dF(y)   = 1-F(x) \\
\int J(x) \; d\tilde{h}(x)   &=&   -\int J(x) \; dF(x) \\
\iint J(y) \; dh(x, y)(y) \; dF(x)  & = &  \int J(x) \; dF(x).
\end{eqnarray*}
So Theorem~\ref{THM:Thm_general_kernels} reproduces the result of \citet{dehling_rooch_taqqu_2013a}. 

\subsubsection{A class of kernels for robust change-point detection}
By Theorem 2.6 in \cite{huber}, an M-estimator of location is robust iff the score function $\Psi$ is bounded. As score functions are typically either nondecreasing or redescending,  they have finite total variation. Examples include the score functions introduced by Andrews, Hampel, Huber, Tukey; see \cite{huber}. The hypothesis of no change corresponds to the location of the differences $X_i-X_j$, $1\leq i\leq k$, $k+1\leq j\leq n$ being 0. This motivates the following class of robust change point statistics: We take the maximum of the two-sample $U$-statistic process with kernel
\begin{equation*}
h(x,y):=\Psi(x-y)
\end{equation*}
where $\Psi$ is a robust score function and hence of bounded total variation. Obviously, conditions \eqref{eq:h1-tv} and \eqref{eq:h2-tv}  of Theorem~\ref{THM:Thm_general_kernels} hold.

\subsubsection{The kernel $h(x,y)=x-y$}
This kernel leads to the classical CUSUM statistic
\begin{align*}
 \frac{1}{n\, d_n} U_{\text{diff}, [\lambda n], n}  &=   \frac{1}{n\, d_n} \sum_{i=1}^{[\lambda n]}\sum_{j=[\lambda n]+1}^{n} (X_i-X_j) \\
   &=   \frac{[\lambda n](n-[\lambda n])}{n\, d_n} \left( \bar{X}_{[\lambda n]} - \bar{X}_{[\lambda n]+1, n} \right),
\end{align*}
where $\bar{X}_{[\lambda n]}$ denotes the arithmetic mean of the first $[\lambda n]$ observations and $\bar{X}_{[\lambda n]+1, n}$ denotes the arithmetic mean of the last $n-[\lambda n]$ observations. Here, the conditions of Theorem~\ref{THM:Thm_general_kernels} are not met, since $h$ is not of bounded total variation. Nevertheless, both integrals occuring in the limit \eqref{EQ:Thm_general_kernels_limit}  exist, and thus we can formally apply Theorem~\ref{THM:Thm_general_kernels}.
In order to show this, we first note that
\[
\tilde{h}(x)    =   \int_\R h(x, y) \; dF(y)   =   \int_\R (x-y)\; d\Phi(y)   =   x.
\]
Moreover $dh(x, y)(y) = d(x-y)(y)= -dy$. 
Both integrals in the limit exist and have the same absolute value, namely
\[
\left| \int_\R J(y) \;dh(x,y)(y) \right|   =   \int_\R \varphi(y) \; dy  =  1.
\]
So, if one ignored that some of the conditions are violated, Theorem~\ref{THM:Thm_general_kernels} would state that $ U_{\text{diff}, [\lambda n], n}$ converges in distribution to the process
\begin{align*}
&\hspace*{-0.7cm} (1-\lambda)Z(\lambda) \int \varphi(x) \; dx - \lambda (Z(1)-Z(\lambda)) \int \left( \int \varphi(y)\; dy \right) \; d\Phi (x)\\
   &=   (1-\lambda)Z(\lambda) - \lambda (Z(1)-Z(\lambda)) \\
   &= Z(\lambda) - \lambda Z(1),
\end{align*}
where $Z(\lambda) = B_{1-D/2}(\lambda)$ denotes the standard fractional Brownian motion with Hurst parameter $H=1-D/2$. A rigorous proof of this result was obtained by \citet{horvath_kokoszka_1997}.

\section{Proofs}

\subsection{Proof of Theorem \ref{THM:2sampleUstatistics_dep}}
The expansion \eqref{EQ:HermiteExpansion2} of $h(x,y)$ in Hermite polynomials converges to $h$ in $L^2(\R^2,\mathcal{N})$. Thus, for independent standard normal random variables $\xi,\eta$, we obtain the series expansion
\[
  h(\xi,\eta)=\sum_{q=m}^\infty \sum_{k,l:k+l=q} \frac{a_{kl}}{k!l!} H_k(\xi) H_l(\eta).
\]
In order to handle $h(\xi_i, \xi_j)$, we face the problem that  any pair $(\xi_i,\xi_j)$ is dependent and has a joint normal distribution with covariance matrix that is not the identity. So first we ensure that the expansion \eqref{EQ:HermiteExpansion2} is nevertheless applicable in our situation. We show first that under condition \eqref{eq:akl_sum}, \eqref{EQ:HermiteExpansion2} converges almost surely to $h(x,y)$. 

By the Cauchy-Schwarz inequality and \eqref{eq:akl_sum} , we obtain
\begin{align*}
E\left[ \sum_{k, l} \left\vert \frac{a_{kl}}{k!\, l!} H_k(\xi_i)H_l(\xi_j) \right\vert\right]   &\leq   \sum_{k, l} \frac{\vert a_{kl} \vert}{k!\, l!} E\left \vert H_k(\xi_i)H_l(\xi_j) \right\vert \\
   &\leq   \sum_{k, l} \frac{\vert a_{kl} \vert}{k!\, l!} \sqrt{E\left[H_k^2(\xi_i)\right] E\left[H_l^2(\xi_j)\right]} \\
   &=   \sum_{k, l} \frac{\vert a_{kl} \vert}{\sqrt{k!\, l!}} <\infty.
\end{align*}
Thus, the series  $\sum_{k, l} \frac{a_{kl}}{k!\, l!} H_k(\xi_i)H_l(\xi_j)$ 
is almost surely absolutely convergent, and the same holds for the series $\sum_{k, l} \frac{a_{kl}}{k!\, l!} H_k(x)H_l(y)$, with respect to any bivariate normal distribution. Since we have $L^2$-convergence to $h(x,y)$, with respect to the bivariate standard normal distribution, and since all nondegenerate normal distributions on $\R^2$ are equivalent, the almost sure limit of $\sum_{k, l} \frac{a_{kl}}{k!\, l!} H_k(\xi_i)H_l(\xi_j)$ is $h(\xi_i,\xi_j)$.
Thus we have
\[
h(\xi_i, \xi_j) = \sum_{\substack{k,l: \\ k+l \geq m}} \frac{a_{kl}}{k!\, l!} H_{k}(\xi_i) H_{l}(\xi_j)
\]
and hence
\[
\sum_{i = 1}^{[\lambda n]}\sum_{j = [\lambda n]+1}^{n} h(\xi_{i}, \xi_{j})  =  \sum_{k,l: k+l \geq m} 
\sum_{i = 1}^{[\lambda n]}\sum_{j = [\lambda n]+1}^{n} \frac{a_{kl}}{k!\, l!} H_{k}(\xi_i) H_{l}(\xi_j).
\]
Thus, we obtain
\begin{align*}
 &\hspace{-0.8cm}
\sup_{0 \leq \lambda \leq 1} \frac{1}{d'_n \, n} \left( \sum_{i = 1}^{[\lambda n]}\sum_{j = [\lambda n]+1}^{n} h(\xi_{i}, \xi_{j}) - \sum_{\substack{k,l: \\ k+l = m}} \sum_{i = 1}^{[\lambda n]} \sum_{j = [\lambda n]+1}^{n} \frac{a_{kl}}{k!\, l!} H_{k}(\xi_i) H_{l}(\xi_j) \right) \\
   &=   \sup_{0 \leq \lambda \leq 1} \frac{1}{d'_n \, n} \sum_{\substack{k,l: \\ k+l \geq m+1}} \sum_{i = 1}^{[\lambda n]}\sum_{j = [\lambda n]+1}^{n} \frac{a_{kl}}{k!\, l!} H_{k}(\xi_i) H_{l}(\xi_j).
\end{align*}
We will show that the term on the right-hand side converges in $L^1$ to $0$. Note first that the supremum here is in fact a maximum, since $\lambda$ appears only in terms of the integer $[\lambda n]$, thus by setting $b= [\lambda n]$ and by the fact that we need to have $[\lambda n]\geq 1$ in order to have a two-sample statistic, we can replace $\sup_{0 \leq \lambda \leq 1}$ by $\max_{1 \leq b \leq n} $. Using the inequality $\max_b |f(b)g(b)| \leq \max_b |f(b)| \, \max_b |g(b)|$ and the Cauchy-Schwarz inequality, we obtain
\begin{align}
  & E \left\vert\frac{1}{d'_n \, n} \sup_{0 \leq \lambda \leq 1} \sum_{\substack{k,l: \\ k+l \geq m+1}} \frac{a_{kl}}{k!\, l!} \sum_{i = 1}^{[\lambda n]}\sum_{j = [\lambda n]+1}^{n} H_{k}(\xi_{i}) H_{l}(\xi_{j})  \right\vert \nonumber \\
   &\leq   E \left[ \frac{1}{d'_n \, n} \sum_{\substack{k,l: \\ k+l \geq m+1}} \frac{|a_{kl}|}{k!\, l!} \max_{1 \leq b \leq n} \left|\sum_{i = 1}^b H_{k}(\xi_{i}) \right| \max_{1 \leq b \leq n} \left| \sum_{j = b+1}^{n} H_{l}(\xi_{j}) \right| \right]  \nonumber \\
   &\leq  \frac{1}{d'_n \, n} \sum_{\substack{k,l: \\ k+l \geq m+1}} \frac{|a_{kl}|}{k!\, l!}  \left( E \left[\max_{1 \leq b \leq n} \left|\sum_{i = 1}^b H_{k}(\xi_{i}) \right| \right]^2  E \left[\max_{1 \leq b \leq n} \left| \sum_{j = b+1}^{n} H_{l}(\xi_{j}) \right|\right]^2 \right)^{1/2}.
   \label{EQ:Varianzen_aufteilen_im_Major_GWS_dep} 
\end{align}

In order to show that this term converges to 0, we need bounds for the expectations of the squared maxima. 
The growth of the partial sum $\sum_{i = 1}^b H_{k}(\xi_{i})$ is determined by the degree $k$ of the Hermite polynomial and the size of the LRD parameter $D \in (0,1)$: For $Dk > 1$ we observe usual SRD behaviour, while for $Dk < 1$ we observe a faster rate of growth, remember \eqref{EQ:Var(sum_Hl(Xj))}. First we consider the SRD case, that is $Dk>1$. Here we have by \eqref{EQ:Var(sum_Hl(Xj))}
\[
 E\left[ \sum_{i=1}^{b} H_{k}(\xi_{i})\right]^2  \leq  C k! \cdot b,
\]
and thus we obtain by stationarity and a maximal inequality of \citet[Theorem 3]{moricz_1976}
\begin{equation}
 \label{EQ:Moricz_bound_SRD} 
 E\left[ \max_{1 \leq b \leq n} \left| \sum_{i=1}^{b} H_{k}(\xi_{i}) \right| \right]^2  \leq  4C k! \cdot n (\log_2 n)^2.
\end{equation}
Here we used the estimate $\log_2 (2n) \leq 2 \log_2 n$ for $n \geq 2$. 

Now we turn to the LRD case, that is $Dk < 1$. Here we have by \eqref{EQ:Var(sum_Hl(Xj))} and the simple estimate $b^{2-Dk} \leq b n^{1-Dk}$ for all $b\leq n$ (and we do not consider any other $b$)
\[
 E\left[ \sum_{i=1}^{b} H_{k}(\xi_{i})\right]^2  \leq  \tilde{C}(k) k! \cdot n^{1-Dk} \, \max_{1 \leq b \leq n} L^k(b) \cdot b,
\]
and thus by the same inequality of \citet[Theorem 3]{moricz_1976}
\begin{equation}
 \label{EQ:Moricz_bound_LRD} 
 E\left[ \max_{1 \leq b \leq n} \left| \sum_{i=1}^{b} H_{k}(\xi_{i}) \right| \right]^2  \leq  4\tilde{C}(k) k! \cdot n^{2-Dk} \max_{1 \leq b \leq n} L^k(b) \cdot (\log_2 n)^2 .
\end{equation}
Note that the same estimates hold in \eqref{EQ:Varianzen_aufteilen_im_Major_GWS_dep} for the sum that starts at $b+1$ because for some $b' \in \{1, \ldots, n\}$
\[
 \max_{1 \leq b \leq n} \left| \sum_{j=b+1}^n H_l(\xi_j) \right|   \stackrel{\mathcal{D}}{=}    \max_{1 \leq b \leq n} \left| \sum_{j=1}^{n-b} H_l(\xi_j) \right|   =    \max_{1 \leq b' \leq n} \left| \sum_{j=1}^{b'} H_l(\xi_j) \right|
\]
where $\stackrel{\mathcal{D}}{=}$ denotes equality in distribution since the $(\xi_i)_{i \geq 1}$ are stationary.

Now depending on the size of $k$ and $l$, both sums in \eqref{EQ:Varianzen_aufteilen_im_Major_GWS_dep} can be SRD or LRD -- and thus they can be bounded by \eqref{EQ:Moricz_bound_SRD} or by \eqref{EQ:Moricz_bound_LRD} --, such that we have to descriminate four cases. 

Case 1: $k,l < 1/D$, i.e. both sums are LRD. In this case, \eqref{EQ:Varianzen_aufteilen_im_Major_GWS_dep} is bounded by
 \begin{align*}
    &\hspace{-0.5cm}   \frac{1}{n^{2-Dm/2} L^{m/2}(n)} \sum_{\substack{k+l \geq m+1 \\ k,l < 1/D}} \frac{|a_{kl}|}{k!\, l!}  \Bigg( \tilde{C}(k) \sqrt{k!} n^{1-Dk/2} \max_{1 \leq b \leq n} L^{k/2}(b) \log_2 n \\
     &\qquad   \cdot \tilde{C}(l) \sqrt{l!} n^{1-Dl/2} \max_{1 \leq b \leq n} L^{l/2}(b) \log_2 n  \Bigg)  \\
    &\leq   \sum_{\substack{k+l \geq m+1 \\ k,l < 1/D}} \frac{|a_{kl}|}{\sqrt{k!\, l!}}  \Bigg( \tilde{C}(k) \tilde{C}(l) \\
    &\qquad    \cdot n^{\frac{D}{2}(m-(k+l))} \max_{1 \leq b \leq n} L^{k/2}(b) \max_{1 \leq b \leq n} L^{l/2}(b) L^{-m/2}(n) (\log_2 n)^2 \Bigg)  
\end{align*}
 Now $n^{\frac{D}{2}(m-(k+l))} = n^{-\varepsilon}$ for some $\varepsilon >0$, and $L^{-m/2}(n)$ and $\log^2_2 n$ are  $o(n^\varepsilon)$ for any $\varepsilon >0$. We will immediately show that also $\max_{1 \leq b \leq n} L^{k/2}(b)$ is $o(n^\varepsilon)$ for any $\varepsilon >0$ and $k\in\N$. Because the summation over $k,l$ is only finite, the sum on the right-hand side is finite, and thus the right-hand side converges to 0.
 
 Now we show that $\max_{1 \leq b \leq n} L^{k/2}(b)$ is $o(n^\varepsilon)$ for any $\varepsilon >0$ and $k\in\N$. When $L$ is slowly varying, $L^{k/2}(x)$ is it as well. So we need to consider
 \begin{align*}
  \max_{1 \leq b \leq n} \frac{L(b)}{n^\varepsilon}   &\leq   \max_{1 \leq b \leq \sqrt{n}} \frac{L(b)}{\sqrt{n}^\varepsilon \sqrt{n}^\varepsilon} +  \max_{\sqrt{n} \leq b \leq n} \frac{L(b)}{n^\varepsilon}   \\
   &\leq   \frac{1}{\sqrt{n}^\varepsilon} \max_{1 \leq b \leq \sqrt{n}} \frac{L(b)}{b^\varepsilon}  +  \max_{\sqrt{n} \leq b \leq n} \frac{L(b)}{b^\varepsilon},
 \end{align*} 
and since $L(b)/b^\varepsilon \to 0$ as $b\to \infty$, $\max_{1 \leq b \leq \sqrt{n}} \frac{L(b)}{b^\varepsilon}$ is bounded and $\max_{\sqrt{n} \leq b \leq n} \frac{L(b)}{b^\varepsilon}$ converges to 0.

Case 2: $k < 1/D$ and $l > 1/D$, i.e. the sum over $i$ is LRD and the sum over $j$ is SRD. In this case, \eqref{EQ:Varianzen_aufteilen_im_Major_GWS_dep} is bounded by
 \begin{align*}
    &\hspace{-0.5cm}   \frac{1}{n^{2-Dm/2} L^{m/2}(n)} \sum_{\substack{k+l \geq m+1 \\ k < 1/D, \, l>1/D}} \frac{|a_{kl}|}{k!\, l!}  \Bigg( C(k) \sqrt{k!} n^{1-Dk/2} \max_{1 \leq b \leq n} L^{k/2}(b) \log_2 n  \\
   &\qquad \cdot  \sqrt{l!} \sqrt{n} \log_2 n  \Bigg)  \\
    &\leq   \sum_{\substack{k+l \geq m+1 \\ k < 1/D, \, l>1/D}} \frac{|a_{kl}|}{\sqrt{k!\, l!}}  \left( C(k) n^{-\frac{1}{2}+\frac{Dm}{2}-\frac{Dk}{2}} \max_{1 \leq b \leq n} L^{k/2}(b)  L^{-m/2}(n) (\log_2 n)^2 \right)  
\end{align*}
 Here, we have summed up some constants in order to keep the expression simple. Now $n^{-\frac{1}{2}+\frac{Dm}{2}-\frac{Dk}{2}} = n^{-\varepsilon}$ for a $\varepsilon >0$, because $\frac{Dm}{2}, \frac{Dk}{2} \in (0, \frac{1}{2})$. $\max_{1 \leq b \leq n} L^{l/2}(b)$, $L^{-m/2}(n)$ and $\log^2_2 n$ are $o(n^\varepsilon)$ for any $\varepsilon >0$ as above, and the sum on the right hand side is finite, because of \eqref{eq:akl_sum} and since the summation over $k$ is only finite. 

Case 3: $k > 1/D$ and $l < 1/D$, i.e. the sum over $i$ is SRD and the sum over $j$ is LRD. In this case, \eqref{EQ:Varianzen_aufteilen_im_Major_GWS_dep} converges to 0 by the same arguments. 

Case 4: $k,l > 1/D$, i.e. both sums are SRD. In this case, \eqref{EQ:Varianzen_aufteilen_im_Major_GWS_dep} is bounded by
 \begin{align*}
    &\hspace{-0.1cm}   \frac{1}{n^{2-Dm/2} L^{m/2}(n)} \sum_{\substack{k+l \geq m+1 \\ k,l > 1/D}} \frac{|a_{kl}|}{k!\, l!}  \left( C \sqrt{k!} \sqrt{n} \log_2 n   \cdot  \sqrt{l!} \sqrt{n} \log_2 n  \right)  \\
    &\leq   C \sum_{\substack{k+l \geq m+1 \\ k,l > 1/D}} \frac{|a_{kl}|}{\sqrt{k!\, l!}}  \left( n^{-1+Dm/2} L^{-m/2}(n) \log^2_2 n \right)  
\end{align*}
 Now $n^{-1+Dm/2} = n^{-\varepsilon}$ for a $\varepsilon >0$, because $\frac{Dm}{2} \in (0, \frac{1}{2})$. $L^{-m/2}(n)$ and $\log^2_2 n$ are $o(n^\varepsilon)$ for any $\varepsilon >0$ as above, and the sum on the right hand side is finite, because of \eqref{eq:akl_sum}. 

Thus, in all four cases, \eqref{EQ:Varianzen_aufteilen_im_Major_GWS_dep} converges to 0, and thus the first statement of the theorem is proved. 

For the proof of the second statement of the theorem, we apply the multivariate non-CLT for LRD processes of Taqqu and Bai (2012), which states that for  any integer $m < 1/D$ 
\begin{align*}
\left( \frac{1}{d_n(1)} \sum_{i=1}^{[\lambda_1 n]} H_1(\xi_i), \frac{1}{d_n(2)} \sum_{i=1}^{[\lambda_2 n]} H_2(\xi_i), \ldots, \frac{1}{d_n(m)} \sum_{i=1}^{[\lambda_m n]} H_m(\xi_i)\right)
\end{align*}
converges in distribution to the $m$-dimensional process
\begin{align*}
\left( Z_1(\lambda_1), Z_2(\lambda_2), \ldots, Z_m(\lambda_m) \right)
\end{align*}
in $D[0,1]^m$, where $d_n^2(k)$ is the usual scaling for the partial sum of the $k$-th Hermite polynomial, as defined in \eqref{eq:d_n}, and $Z_k$ denotes the $k$-th order Hermite process, as defined in \eqref{eq:Z-Rep}, driven by the same Wiener process $W$.

We now consider the cases where $k+l=m$. By the multivariate non-CLT and the continuous mapping theorem, we obtain
\begin{align*}
  &\hspace{-0.4cm} 
\frac{1}{d'_n \, n} \sum_{i = 1}^{[\lambda n]} \sum_{j = [\lambda n]+1}^{n} H_{k}(\xi_{i}) H_{l}(\xi_{j})  \\
   &=   n^{-2+D m/2} L(n)^{-m/2} \sum_{i = 1}^{[\lambda n]}H_{k}(\xi_{i}) \left( \sum_{j = 1}^{n}  H_{l}(\xi_{j}) - \sum_{j = 1}^{[\lambda n]} H_{l}(\xi_{j}) \right) \\
   &\stackrel{\mathcal{D}}{\longrightarrow}   c_k^{1/2} Z_k(\lambda)  \cdot c_l^{1/2} (Z_l(1)-Z_l(\lambda))
\end{align*}
uniformly in $\lambda \in [0,1]$. 
\hfill $\Box$

\subsection{Proof of Theorem \ref{THM:Thm_general_kernels}}
We will express the two-sample U-statistic process as a functional of the empirical process of the observations $(X_i)_{i\geq 1}$ and apply the empirical process non-CLT of Dehling and Taqqu (1989),
which states that for $Dm<1$,  as $n\to\infty$ 
\begin{equation} 
 \label{EQ:dehling_taqqu_1989_GWS} 
 \left ( \frac{[\lambda n]}{d_n} (F_{[\lambda n]}(x)-F(x))\right)_{x\in [-\infty, \infty], \lambda \in [0,1]}   \stackrel{\mathcal{D}}{\longrightarrow}   \left ( \frac{J_m(x)}{m!} Z_m (\lambda) \right)_{x\in [-\infty, \infty], \lambda \in [0,1]}.
\end{equation}

We denote by $F_{k,n}$ the empirical distribution function of the observations $X_k,\ldots,X_n$, i.e.
\[
  F_{k,n}(x)=\frac{1}{n-k+1} \sum_{i=k}^n 1_{ \{ X_i\leq x\} }.
\]
Dehling, Rooch and Taqqu (2013a) have shown that, applying 
the Dudley-Wichura version of Skorohod's representation theorem \citep[Th.\ 2.3.4]{shorack_wellner_1986}  to (\ref{EQ:dehling_taqqu_1989_GWS}),  we may assume without loss of generality that
\begin{align}
\label{EQ:DehlingTaqqu_SkorokhodVersion_KapExtensions}
&\sup_{\lambda, x}  \left| d_n^{-1} [n\lambda] (F_{[\lambda n]}(x) - F(x)) 
 -J(x) Z(\lambda)\right|   \longrightarrow   0 \quad a.s. \\
& \sup_{\lambda, x} \left| d_n^{-1} (n-[n \lambda]) (F_{[n\lambda]+1,n}(x)-F(x))  -  J(x) (Z(1)-Z(\lambda))\right|   \longrightarrow   0 \quad a.s.\ .
 \label{EQ:DehlingTaqqu_SkorokhodVersion2_KapExtensions}
\end{align}

Next, we obtain some auxiliary results concerning the kernel $h$, and the integrals that occur in the limit 
 \eqref{EQ:Thm_general_kernels_limit}. 
First, note that  \eqref{eq:h1-tv} and \eqref{eq:h2-tv} imply that the kernel $h$ is bounded. Moreover,
\eqref{eq:h1-tv} implies that $\tilde{h}$ is of bounded total variation, since
for any  $x_0<x_1<\ldots<x_k$
\begin{eqnarray*}
\sum_{i=1}^k |\tilde{h} (x_i) -h(x_{i-1})|
 &=& \sum_{i=1}^k \left|\int \left(h(x_i,y)-h(x_{i-1},y)\right) dF(y)\right| \\
 &\leq &\int \sum_{i=1}^k \left(\left| h(x_i,y)-h(x_{i-1},y)  \right|\right) dF(y) \leq c.
\end{eqnarray*}
Thus $\|\tilde{h}\|_{TV} \leq c$ and, as a consequence, $\tilde{h}$ is bounded. By definition \eqref{eq:J_k}, the function $J(x)$ is bounded, and thus 
\begin{equation}
  \left| \int J(x) d\tilde{h}(x)\right| \leq \|J\|_\infty c,
\label{eq:int_J1}
\end{equation}
where $\|J\|_\infty$ denotes the supremum norm of $J$. In the same way, we obtain 
\begin{equation}
 \int \left( \int J(y) dh(x,y)(y) \right) dF(x) \leq \int \|J\|_\infty c\; dF(x) \leq \|J\|_\infty c.
\label{eq:int_J2}
\end{equation}
We now express the two-sample U-statistic as a functional of the empirical process, and obtain
\begin{align*}
&\hspace*{-1cm} \sum_{i=1}^{[\lambda n]} \sum_{j=[\lambda n]+1}^{n} \left( h(X_i, X_j) - \iint h(x, y)\; dF(x)dF(y)\right)  \nonumber \\
   &=   [\lambda n](n-[n\lambda]) \Bigg\{ \int \left( \int h(x,y)\;dF_{[\lambda n]+1,n}(y) \right) \; dF_{[\lambda n]}(x) \nonumber \\
   &\qquad - \iint h(x, y)\; dF(x)dF(y) \Bigg\} \nonumber \\
   &=   [\lambda n](n-[n\lambda]) \Bigg\{ \int \left( \int h(x,y)\;d(F_{[\lambda n]+1,n}-F)(y) \right) \; dF_{[\lambda n]}(x) \nonumber \\
   &\qquad + \int \left( \int h(x,y)\;dF(y) \right) \; d(F_{[\lambda n]}-F)(x) \Bigg\}.
\end{align*}
Next, we integrate by parts in order to get the ``$F_n-F$" terms as integrands and the deterministic terms as integrators. Regarding the first integral, we obtain
\begin{align}
&\hspace*{-1cm}   \int h(x,y)\;d(F_{[\lambda n]+1,n}-F)(y) \nonumber \\
   &=   \left[ h(x,y) \cdot (F_{[\lambda n]+1,n}-F)(y) \right]_{y=-\infty}^{\infty} 
- \int (F_{[\lambda n]+1,n}-F)(y) \; dh(x,y)(y) \nonumber\\
 &= - \int (F_{[\lambda n]+1,n}-F)(y) \; dh(x,y)(y), \label{EQ:GeneralKernel_IntByParts_1} 
\end{align}
where the boundary term vanishes, since $h$ is bounded. Regarding the second integral, we obtain
\begin{align}
\int \!\! \! \left(\! \int h(x,y) dF(y) \right) d(F_{[\lambda n]}-F)(x) 
   &=  \left[ \tilde{h}(x) (F_{[\lambda n]}-F)(x) \right]_{x=-\infty}^{\infty} - \int (F_{[\lambda n]}-F)(x) \; d\tilde{h}(x)\nonumber \\
&=- \int (F_{[\lambda n]}-F)(x) \; d\tilde{h}(x),\label{EQ:GeneralKernel_IntByParts_2}
\end{align}
where the boundary term vanishes by boundedness of $\tilde{h}$.
In total, we thus obtain
\begin{align}
&\hspace*{-1cm}   \frac{1}{n \, d_n} \sum_{i=1}^{[\lambda n]} \sum_{j=[\lambda n]+1}^{n} \left( h(X_i, X_j) - \iint h(x, y)\; dF(x)dF(y)\right)  \nonumber\\
   &=   \frac{[\lambda n](n-[n\lambda])}{n\, d_n} \Bigg\{ \int \left( -\int (F_{[\lambda n]+1,n}-F)(y) \; dh(x,y)(y) \right) \; dF_{[\lambda n]}(x) \nonumber \\
   &\qquad - \int (F_{[\lambda n]}-F)(x) \; d\tilde{h}(x) \Bigg\}. \label{EQ:Generel_kernel_Limit0}
\end{align}
We now consider both terms on the right-hand side of \eqref{EQ:Generel_kernel_Limit0} separately, and show that they converge  to the limit given in \eqref{EQ:Thm_general_kernels_limit}. Regarding the second term, 
we obtain
\begin{align}
&\hspace*{-1cm}   \frac{[\lambda n](n-[n\lambda])}{n \, d_n} \int (F_{[\lambda n]}-F)(x) \; d\tilde{h}(x) - (1-\lambda)\int J(x)Z(\lambda) \; d\tilde{h}(x)  \nonumber \\
   &=   \frac{n-[n\lambda]}{n} \int \left( d_n^{-1} [\lambda n](F_{[\lambda n]}-F)(x) - J(x)Z(\lambda)\right) \; d\tilde{h}(x) \nonumber \\
   &\qquad +\left( \frac{n-[\lambda n]}{n} - (1-\lambda) \right) \int J(x)Z(\lambda) \; d\tilde{h}(x) \label{EQ:Generel_kernel_Limit1} 
\end{align}
The first summand in \eqref{EQ:Generel_kernel_Limit1}  converges to zero because of \eqref{EQ:DehlingTaqqu_SkorokhodVersion_KapExtensions} and the bounded total variation of $\tilde{h}$.  The second summand converges by \eqref{eq:int_J1},
since $\sup_{0\leq \lambda \leq 1} |(n-[n\lambda])/n-(1-\lambda)| \to 0$.

Regarding the first  term on the right-hand side of \eqref{EQ:Generel_kernel_Limit0}, we get
\begin{align}
&\frac{[\lambda n](n-[n\lambda])}{n\, d_n} \int \left( \int (F_{[\lambda n]+1,n}-F)(y) \; dh(x,y)(y) \right) \; dF_{[\lambda n]}(x) \nonumber\\
   &\qquad \qquad \qquad- \lambda \int \left( \int J(y)(Z(1)-Z(\lambda))\; dh(x, y)(y) \right) \; dF(x) \nonumber  \\
   &=   \frac{[n\lambda]}{n} \iint \left\{ d_n^{-1}(n-[n\lambda]) (F_{[n\lambda]+1,n}-F)(y) 
  - J(y)(Z(1)-Z(\lambda)) \right\} dh(x, y)(y) \; dF_{[n\lambda]}(x)\nonumber \\  
   &\qquad  \qquad \qquad+ \frac{[n\lambda]}{n} (Z(1)-Z(\lambda)) \int \left( \int J(y) \; dh(x, y)(y) \right) \; d(F_{[n\lambda]}-F)(x)    \nonumber \\
   &\qquad \qquad \qquad +\left( \frac{[n\lambda]}{n}-\lambda \right) (Z(1)-Z(\lambda)) \int\! \left( \!\int J(y) \; dh(x, y)(y) \right)  dF(x) \label{EQ:Generel_kernel_Limit2} 
\end{align}
All three terms on the right-hand side converge to zero. For the last term this is a consequence of $\sup_{0\leq \lambda \leq 1} |[n\lambda]/n-\lambda| \to 0$, and of \eqref{eq:int_J2}. Convergence of the first term follows from \eqref{EQ:DehlingTaqqu_SkorokhodVersion2_KapExtensions}, together with the estimate
\begin{align*}
 & \left| \int \left\{ d_n^{-1}(n-[n\lambda]) (F_{[n\lambda]+1,n}-F)(y) 
  - J(y)(Z(1)-Z(\lambda)) \right\} dh(x, y)(y) \right| \\
 & \leq c\sup_{\lambda,x} \left| \left\{ d_n^{-1}(n-[n\lambda]) (F_{[n\lambda]+1,n}-F)(y) 
  - J(y)(Z(1)-Z(\lambda)) \right\}  \right|.
\end{align*}
The second term in \eqref{EQ:Generel_kernel_Limit2} can be written as
\begin{align*}
&\hspace*{-0,7cm}   \frac{[n\lambda]}{n} \int \left( \int J(y) \; dh(x, y)(y) \right) \; d(F_{[n\lambda]}-F)(x)\\
   &=   \frac{1}{n} \sum_{i=1}^{[\lambda n]} \int J(y)\; dh(X_i, y)(y) - E\left[ \int J(y)\; dh(X_1, y)(y) \right]\\
 &= \frac{1}{n} \sum_{i=1}^{[n\lambda]} (f(X_i)-Ef(X_i)),
\end{align*}
where  $f(x)=\int J(t)\;dh(x,t)(t)$. As $E|f(X_i)| = \int |f(x)| dF(x) < \infty$, we may apply the ergodic theorem to obtain convergence of the right hand side to zero, almost surely.
\hfill $\Box$

\end{document}